\documentclass[11pt]{amsart}
\usepackage{fullpage,cite,hyperref}
\usepackage{amssymb,amsmath,amsthm,mathtools,extpfeil}
\usepackage[all,cmtip]{xy}
\usepackage{tikz}
\usepackage{rotating,enumitem}
\newtheorem{thm}[equation]{Theorem}
\newtheorem*{thm*}{Theorem}
\newtheorem{lem}[equation]{Lemma}
\newtheorem{prop}[equation]{Proposition}
\newtheorem*{prop*}{Proposition}
\newtheorem{cor}[equation]{Corollary}

\theoremstyle{definition}
\newtheorem*{defn}{Definition}
\newtheorem{rmk}[equation]{Remark}
\newtheorem{ex}[equation]{Example}

\numberwithin{equation}{section}

\DeclareMathOperator{\rat}{loc}
\DeclareMathOperator{\id}{id}

\DeclareMathOperator{\vol}{vol}

\DeclareMathOperator{\rk}{rk}

\DeclareMathOperator{\Hom}{Hom}

\DeclareMathOperator{\Lip}{Lip}

\newcommand{\ph}{\varphi}

\begin{document}
\title{Integral and rational mapping classes}
\author{Fedor Manin and Shmuel Weinberger}
\maketitle

\begin{abstract}
  Let $X$ and $Y$ be finite complexes.  When $Y$ is a nilpotent space, it has a
  rationalization $Y \to Y_{(0)}$ which is well-understood.  Early on it was
  found that the induced map $[X,Y] \to [X,Y_{(0)}]$ on sets of mapping classes
  is finite-to-one.  The sizes of the preimages need not be bounded; we show,
  however, that as the complexity (in a suitable sense) of a rational mapping
  class increases, these sizes are at most polynomial.  This ``torsion''
  information about $[X,Y]$ is in some sense orthogonal to rational homotopy
  theory but is nevertheless an invariant of the rational homotopy type of $Y$ in
  at least some cases.  The notion of complexity is geometric and we also prove a
  conjecture of Gromov \cite{GrMS} regarding the number of mapping classes that
  have Lipschitz constant at most $L$.
\end{abstract}

\section{Introduction}
One of the great successes of homotopy theory is the complete algebraicization of
rational homotopy theory by Quillen \cite{Qui} and Sullivan \cite{SulLong}.  In
particular, while the main objects of study are best understood as infinite
complexes, the finiteness theorem of Sullivan and Hilton--Mislin--Roitberg,
quoted below, and related results allow the ideas to be applied to the homotopy
theory of based maps between finite complexes $X \to Y$, where $Y$ is a nilpotent
space.

To each nilpotent complex $Y$ of finite type\footnote{A CW complex is of finite
  type if it is homotopy equivalent to one with finitely many cells in every
  dimension.} is associated a (functorially constructed) rationalization
$Y_{(0)}$, characterized by the condition that the map $Y \to Y_{(0)}$ induces an
isomorphism on $\pi_i({-}) \otimes \mathbb{Q}$.  In this context, the finiteness
theorem says the following:
\begin{thm*}[10.2(i) in \cite{SulLong} or II.5.4 in \cite{HMR}]
  If $X$ is a finite complex, then the map between (based or unbased) sets of
  homotopy classes $[X,Y] \to [X,Y_{(0)}]$ is finite-to-one.
\end{thm*}

This paper is devoted to understanding this more quantitatively.  If $X$ is the
sphere, say, then the number of preimages of a rational map is independent of the
map---when it is nonzero---since it is the cardinality of the kernel of the map
$\pi_i(Y) \to \pi_i(Y) \otimes \mathbb{Q}$.  However, even for as simple a target
as $S^4$ there are examples where the cardinalities of these fibers can be
unbounded.  If $X=S^3 \times S^4$, the sizes depend on the degree of the map
restricted to the $S^4$ factor.

The precise statement that we will make is that the sizes of these preimages
grows like a polynomial in the complexity of the homotopy class.  But this forces
the question of defining complexity.

There are two solutions to this new problem, both suggested by Gromov's
well-known paper \cite{GrDil}.
\begin{enumerate}
\item Replace $X$ and $Y$ by manifolds with boundary, and consider the norms that
  the maps induce on sufficiently large finite dimensional algebras of
  differential forms.  Unfortunately, unlike the minimal model, which is unique
  up to isomorphism (see section 2 for relevant concepts in the homotopy theory
  of commutative differential graded algebras), the algebras of differential
  forms are not.  We will see that, in consequence, although the notion of
  polynomially bounded is well-defined, the degree of the polynomial is not.
\item Fix (possibly piecewise or cell-wise) Riemannian metrics on $X$ and $Y$,
  and view the Lipschitz constant of $f:X \to Y$ as the complexity of the map.
  Minimizing this over representatives gives us a measure of complexity for
  homotopy classes.  For example, the complexity of a degree $d$ map
  $S^n \to S^n$ is roughly $d^{1/n}$.
\end{enumerate}
The notion in (1) does not require the homomorphism to be integral.  One just
asks for a map of DGAs with certain properties and considers the impacts on
norms.  For integral classes, this quantity can be estimated from the map itself
and bounded in terms of its Lipschitz constant.

On the other hand, the notion in (2) is only defined for integral mapping classes
and moreover depends a priori on the metric.  However, any homotopy equivalence
between finite metric complexes is homotopic to a Lipschitz homotopy
equivalence in the obvious sense.  Hence asymptotics with respect to it (up to
a multiplicative constant) are actually homotopy invariants.

The precise interconnection between (1) and (2) is complex.  Gromov noted in
\cite{GrDil}, and J.~Maher showed in more detail in his unpublished thesis
\cite{Maher}, that the rational invariants of Lipschitz maps to a nilpotent
complex (and therefore the notion in (1)) are bounded by a polynomial in the
Lipschitz constant.  In fact, it turns out that the minimal Lipschitz constant is
likewise polynomially bounded by the notion in (1); indeed, in \cite{PCDF} the
first author shows this by way of a purely rational notion of dilatation which
is equivalent to (2) up to a multiplicative constant.  This points to (2) as the
``correct'' notion of complexity for rational mapping classes even from an
algebraic-topological point of view.

We will show that 
\begin{thm} \label{thm:tg}
  The size of the preimages of maps in $[X, Y] \to [X, Y_{(0)}]$ is bounded by a
  polynomial in the complexity (in either sense).
\end{thm}
There is another phenomenon that is dual to this that also needs to be
considered, namely a statement about the density of the image of $[X, Y]$ in
$[X, Y_{(0)}]$.  The image is always discrete, but the density, i.e.~the number
of image points in a ``ball of radius 1'', can grow as one moves farther from the
zero map.  However, this density turns out to be polynomial as well.

These conclusions can be summarized by the following theorem, proved by Gromov
when $X$ is a sphere in \cite{GrDil} and conjectured by him in
\cite[Ch.~7]{GrMS}.
\begin{thm} \label{thm:g}
  The number of homotopy classes of maps $[X, Y]$ that have a representative with
  Lipschitz constant at most $L$ is bounded by a polynomial in $L$.
\end{thm}
However, in section 3, we will see that (contrary to the speculation in Gromov's
book \cite[p.~358]{GrMS}) the number is not necessarily asymptotic to a
polynomial.  We give an example that has an extra $\log(L)$ factor.

In all cases we work with based maps, but the corresponding results for unbased
maps follow easily.  Moreover, our examples are all in the realm of simply
connected spaces, where these notions are equivalent.

\subsection{Rational invariance} \label{S:ratinv}

The techniques in the proofs of Theorems \ref{thm:tg} and \ref{thm:g} are a
variation of the work of Sullivan.  It is natural to ask whether the growth rates
in these problems, which we call \emph{torsion growth}\footnote{We hope not
  confusingly, since there is only growth in situations where the mapping set
  does \emph{not} have a group structure.} and \emph{growth}, respectively, are
actually invariants of the rational homotopy type of $Y_{(0)}$ and not just the
integral homotopy type of $Y$.  It is unclear whether this is true in general,
but we prove the following partial result:
\begin{thm} \label{thm:posw}
  Let $X$ and $Y$ be finite metric complexes with $Y$ simply connected.  If $Y$
  (resp.~$X$) is a \emph{space with positive weights}, then the asymptotic
  behavior of the growth $g_{[X,Y]}$ and the torsion growth $\mathrm{tg}_{[X,Y]}$
  depends only on the rational homotopy type of $Y$ (resp.~$X$).
\end{thm}
Having positive weights is a technical condition on the rational homotopy type of
a simply connected space first introduced by Morgan and
Sullivan\footnote{According to \cite{BMSS}, although this class of spaces was
  studied earlier by Mimura and Toda \cite{MT}.}.  The main property of such
spaces is that they have a large family of ``telescoping'' automorphisms.  Many
naturally occurring simply connected spaces have positive weights; in particular,
all of our examples do.  Therefore, for example, the $L^8\log L$ growth we show
for $[(S^3 \times S^4)^{\#2},S^4]$ is a rational homotopy invariant.  On the other
hand, this theorem does not give any information about non-simply connected
nilpotent spaces.

For more general spaces, we can only say something much weaker:
\begin{prop} \label{prop:Fto1}
  Given a rational homotopy equivalence $Y \to Z$ between finite nilpotent
  complexes, for any finite complex $X$, the induced map $[X,Y] \to [X,Z]$ is
  uniformly finite-to-one; i.e., preimages of classes have bounded size.
\end{prop}
This is insufficient even to prove rational invariance of torsion growth because,
for example, there may be classes in $[X,Z_{(0)}]$ that have quickly growing
preimages in $[X,Z]$, but no preimages at all in $[X,Y]$.

The difficulties may be number-theoretic: in general, the integral classes of
homomorphisms $X \to Y$ are integer points of an arbitrarily complicated
algebraic variety cut out by the differentials of $X$ and $Y$.  It is unclear
what rationally invariant estimates can be found in general.

\subsection{Structure of the paper}
Section 2 provides the background about DGAs and their connection to
rational homotopy theory.  We recommend the (impatient) reader skip partway
through to section 3, which gives examples of the various phenomena that this
paper grapples with, and then go back to complete section 2.  In section 4 we
explain methods (1) and (2) for defining sizes of maps, and in section 5, we
prove our main theorem by combining the ideas of section 4 with Sullivan's
inductive method.  Finally the last section addresses the rational invariance
problem.

\subsection{Acknowledgements}
The first author would like to thank MSRI for its hospitality during
September--October 2017, and both authors would like to thank the Israel
Institute for Advanced Studies at Hebrew University for its hospitality during
(the rest of) Fall 2017.  We would like to thank two anonymous referees for a
number of useful comments drawing our attention to errors and disfluencies in the
presentation.

\section{Homotopy theory of DGAs}

In this section we sketch out the homotopy theory of differential graded
algebras, following the treatment in \cite[Ch.~IX and X]{GrMo}.  Their relatively
explicit formulation helps us obtain quantitative bounds.  This justifies a
thorough exposition as the formalism may differ from more abstract modern
treatments.

A \emph{(commutative) differential graded algebra} (DGA) will always denote a
cochain complex of $\mathbb{Q}$- or $\mathbb{R}$-vector spaces equipped with a
graded commutative multiplication which satisfies the (graded) Leibniz
rule.\footnote{We use the abbreviation ``DGA'' for ``differential graded
  algebra'', following \cite{GrMo} and \cite{FHT}.  In other areas this
  abbreviation may be reserved for augmented algebras.  Minimal algebras have a
  natural augmentation which sends indecomposables to zero, but cochain algebras
  generally do not.}  The prototypical example of an $\mathbb{R}$-DGA is the
algebra of smooth forms on a manifold or piecewise smooth forms on a simplicial
complex.  On a simplicial complex $X$, one can also define a $\mathbb{Q}$-DGA
$A^*X$ of \emph{polynomial forms}; see \cite[Ch.~VIII]{GrMo} for a detailed
exposition.  In the rest of the section we will denote $\mathbb{Q}$ or
$\mathbb{R}$ by $\mathbb{F}$.

The cohomology of a DGA is the cohomology of the underlying cochain complex.  The
relative cohomology $H^n(\ph)$ of a DGA homomorphism
$\ph:\mathcal{A} \to \mathcal{B}$ is defined to be the cohomology of the cochain
complex
$$C^n(\ph)=\mathcal{A}^n \oplus \mathcal{B}^{n-1}$$
with the differential given by $d(a,b)=(da,\ph(a)-db)$.  This cohomology fits, as
expected, into the obvious exact sequence involving $H^*(\mathcal{A})$ and
$H^*(\mathcal{B})$.

Given a finite-dimensional vector space $V$, we write $H^*(\mathcal{A};V^*)$,
where $V^*$ is the dual of $V$, for the cohomology of the cochain complex
$\Hom(V,\mathcal{A})$.  By the universal coefficient theorem, this is naturally
isomorphic to $\Hom(V,H^*(\mathcal{A}))$, but we will refer to individual
cochains in the former format.

A \emph{weak equivalence} between DGAs $\mathcal{A}$ and $\mathcal{B}$ is a
homomorphism $\mathcal{A} \to \mathcal{B}$ which induces an isomorphism on
cohomology.

An algebra $\mathcal{A}$ is \emph{simply connected} if $\tilde H^0(\mathcal{A})=
H^1(\mathcal{A})=0$.  If $\mathcal{A}$ is simply connected and of
\emph{finite type} (i.e.\ it has finite-dimensional cohomology in every degree)
then it has a \emph{minimal model}: a weak equivalence
$m_{\mathcal{A}}:\mathcal{M_A} \to \mathcal{A}$ where $\mathcal{M_A}$ is freely
generated as an algebra by finite-dimensional vector spaces $V_n$ in degree $n$,
written
$$\mathcal{M_A}=\bigwedge_{n=2}^\infty V_n,$$
and the differential satisfies
$$dV_n \subseteq \bigwedge_{k=2}^{n-1} V_k.$$
In other words, $\mathcal{M_A}$ can be built up via a sequence of
\emph{elementary extensions}
$$\mathcal{M_A}(n+1)=\mathcal{M_A}(n) \otimes {\wedge} V_{n+1}$$
with a differential extending that on $\mathcal{M_A}(n)$, starting with
$\mathcal{M_A}(1)=\mathbb{Q}$ or $\mathbb{R}$.  We refer to elements of the $V_n$
as \emph{indecomposables}.  We will often define finitely generated free DGAs by
indicating the degree of generators as superscripts in parentheses: $a^{(3)}$
means that $a$ is an indecomposable generator in degree 3.

In particular, if $Y$ is a manifold or simplicial complex which is simply
connected and of finite cohomological type, the algebras of forms
$\mathcal{A}=A^*Y$ or $\Omega^*Y$ each have minimal models, both of which we will
call $m_Y:\mathcal{M}_Y^* \to \mathcal{A}$ (this notational confusion will not
cause us any problems).  This models the Postnikov tower of $Y$: each
$V_n \cong \Hom(\pi_n(Y),\mathbb{F})$ and the differential on $V_n$ is dual to
the $k$-invariant of the fibration $Y_{(n)} \to Y_{(n-1)}$.  This is shown
inductively via the obstruction theory discussed below.

More generally, suppose $Y$ is a \emph{nilpotent space}: that is, its fundamental
group is nilpotent and acts nilpotently on the higher homotopy groups.  Then it
still has a minimal model $m_Y:\mathcal{M}_Y^* \to \mathcal{A}$, in the sense
that it is built as a limit of extensions
$\mathcal{M}_Y^*(n+1)=\mathcal{M}_Y^*(n) \otimes {\wedge}V_{n+1},$
but now $n$ can no longer denote the degree of the extension as the degrees
$$1 \leq \deg V_1 \leq \cdots \leq \deg V_n \leq \deg V_{n+1} \leq \ldots$$
need not be strictly increasing.  In other words, the $k$th Postnikov stage
yields a finite sequence of elementary extensions which correspond to a
decomposition of the Postnikov stage $K(\pi_k(Y),k) \to Y_{(k)} \to Y_{(k-1)}$ into
a sequence of principal fibrations.

Even more generally, we say $\mathcal{A}$ is \emph{geometric for} a space $Y$ if
there is a weak equivalence $\mathcal{A} \to A^*Y$ or $\Omega^*Y$.

\subsection{Obstruction theory}

Given a principal fibration $K(\pi,n) \to E \to B$, obstruction theory gives an
exact sequence of based sets
$$H^n(X;\pi) \to [X,E] \to [X,B] \to H^{n+1}(X;\pi)$$
of sets of based homotopy classes; see e.g.\ \cite[Prop.~14.3]{GrMo}.  Moreover,
over a given map $f:X \to B$, there is an exact sequence of groups
$$\cdots \to \pi_1(E^X,\tilde f) \to \pi_1(B^X,f) \to H^n(X;\pi) \to
\{\text{lifts of }f\} \to 0,$$
where the set of lifts is not a group but has a transitive action by
$H^n(X;\pi)$.

We now give DGA versions of these statements.  First define homotopy of DGA
homomorphisms as follows: $f,g:\mathcal{A} \to \mathcal{B}$ are homotopic if
there is a homomorphism
$$H:\mathcal{A} \to \mathcal{B}\otimes {\wedge}(t^{(0)},dt^{(1)})$$
such that $H|_{\substack{t=0\\dt=0}}=f$ and $H|_{\substack{t=1\\dt=0}}=g$.  We think of
$\wedge(t,dt)$ as an algebraic model for the unit interval and this notion as an
abstraction of the map induced by an ordinary smooth or simplicial homotopy.  In
particular, it defines an equivalence relation \cite[Cor.~10.7]{GrMo}.

We also introduce some notation which is useful for constructing homotopies
between DGA homomorphisms.  For any DGA $\mathcal A$, define an operation
$\int_0^t:\mathcal A \otimes {\wedge}(t,dt) \to
\mathcal A \otimes {\wedge}(t,dt)$ of degree $-1$ by
$${\textstyle\int_0^t a \otimes t^i}=0,\hspace{4em}
  {\textstyle\int_0^t a \otimes t^idt}=(-1)^{\deg a}a \otimes \frac{t^{i+1}}{i+1}$$
and an operation $\int_0^1:\mathcal A \otimes {\wedge}(t,dt) \to
\mathcal A$ of degree $-1$ by
$${\textstyle\int_0^1 a \otimes t^i}=0,\hspace{4em}
  {\textstyle\int_0^1 a \otimes t^idt}=(-1)^{\deg a}\frac{a}{i+1}.$$
These provide a formal analogue of fiberwise integration; in particular, they
satisfy the identities
\begin{align}
  d\bigl({\textstyle\int_0^t} u\bigr)+{\textstyle\int_0^t} du &=
  u-u|_{\substack{t=0\\dt=0}} \otimes 1 \\
  d\bigl({\textstyle\int_0^1} u\bigr)+{\textstyle\int_0^1} du &=
  u|_{\substack{t=1\\dt=0}}-u|_{\substack{t=0\\dt=0}}. \label{eqn:I01}
\end{align}

Now we state the main lemma of obstruction theory, which gives the conditions
under which a map can be extended over an elementary extension.
\begin{prop}[10.4 in \cite{GrMo}] \label{extExist}
  Let $\mathcal{A}\otimes {\wedge}V$ be a degree $n$ elementary extension of a
  DGA $\mathcal{A}$.  Suppose we have a diagram of DGAs
  $$\xymatrix{
    \mathcal{A} \ar[r]^f \ar@{^{(}->}[d] & \mathcal{B} \ar[d]^h \\
    \mathcal{A}\otimes {\wedge}V \ar[r]^g & \mathcal{C}
  }$$
  with $g|_{\mathcal{A}} \simeq hf$ by a homotopy $H:\mathcal{A} \to \mathcal{C}
  \otimes {\wedge}(t,dt)$.  Then the map $O:V \to
  \mathcal{B}^{n+1} \oplus \mathcal{C}^n$ given by
  $$O(v)=\left(f(dv), g(v)+{\textstyle\int_0^1 H(dv)}\right)$$
  defines an obstruction class $[O] \in H^{n+1}(h:\mathcal{B}\to\mathcal{C};V^*)$
  to producing an extension $\tilde f:\mathcal{A}\otimes {\wedge}V \to
  \mathcal{B}$ of $f$ with $h \circ \tilde f \simeq g$ via a homotopy $\tilde H$
  extending $H$.
\end{prop}
When the obstruction vanishes, there are maps
$(b,c): V \to (\mathcal{B}^n,\mathcal{C}^{n-1})$ such that $d(b,c)=O$, i.e.
\begin{align*}
  db(v) &= f(dv) \\
  dc(v) &= h\circ b(v)-g(v)-{\textstyle\int_0^1 H(dv)}.
\end{align*}
Then for $v \in V$ we can set $\tilde f(v)=b(v)$ and
$$\tilde H(v)=h \circ \tilde f(v)+{\textstyle\int_0^t H(dv)}+d(c(v) \otimes t).$$
This gives a specific formula for the extension.

There is also a relative version of this proposition, as in
\cite[Prop.~10.5]{GrMo}.  This can be used to prove the following.
\begin{prop} \label{14.4ext}
  Let $\mathcal{A} \otimes {\wedge V}$ be a degree $n$ elementary extension of a
  DGA $\mathcal{A}$.  Let $h:\mathcal{B} \to \mathcal{C}$ be a surjection of
  DGAs, and $\mathcal{A} \otimes {\wedge V} \xrightarrow{\ph} \mathcal{C}$ be a
  map.  Then there is an exact sequence of based sets
  $$H^n(h:\mathcal{B} \to \mathcal{C};V^*) \to
  [\mathcal{A} \otimes {\wedge V},\mathcal{B}]_\ph \to
  [\mathcal{A},\mathcal{B}]_{\ph|_\mathcal{A}} \to
  H^{n+1}(h:\mathcal{B} \to \mathcal{C};V^*)$$
  of homotopy classes of lifts of $\ph$.  Moreover, for every lift
  $\psi:\mathcal{A} \to \mathcal{B}$ of $\ph|_{\mathcal{A}}$, there is an exact
  sequence of groups and a set
  $$[\mathcal{A},\mathcal{B} \otimes {\wedge e^{(1)}}]_{\psi \otimes 1}
  \xrightarrow{\mathcal{O}} H^n(h:\mathcal{B} \to \mathcal{C};V^*) \to
  \left\{\begin{array}{c}\text{extensions of }\psi\\
  \text{in }[\mathcal{A} \otimes \wedge V,\mathcal{B}]_\ph\end{array}\right\}
  \to 0.$$
  Here in the first term, we are looking at lifts of $\psi \otimes 1$ as a map
  $$\mathcal{A} \to (\mathcal{B}\otimes{\wedge e})/\ker h \otimes (e),$$
  that is, self-homotopies of $\psi$ which project to
  $\ph|_{\mathcal{A}} \otimes 1$, and the obstruction $\mathcal{O}$ sends
  $$\psi+\eta \otimes e \mapsto (\eta d|_V,0).$$
\end{prop}
This is a mild extension of \cite[Prop.~14.4]{GrMo} and is proved in essentially
the same way.

In the case of spaces, a principal fibration $K(\pi,n) \to E \to B$ induces a
fibration $K(\pi,n)^X \to E^X \to B^X$ of spaces of based maps, for any
CW-complex $X$.  The homotopy exact sequence of this fibration is
\begin{multline} \label{ES:spaces}
  \cdots \to H^{n-k}(X;\pi) \to \pi_k(E^X,\tilde f) \to \pi_k(B^X,f)
  \xrightarrow{\iota_k} H^{n-k+1}(X;\pi) \to \cdots \\
  \cdots \to \pi_1(B^X,f) \xrightarrow{\iota_1} H^n(X;\pi) \to
  \left\{\begin{array}{c}\text{homotopy classes}\\\text{of lifts of }f\end{array}
  \right\} \to 0,
\end{multline}
where $\tilde f:X \to E$ is an arbitrary lift of the map $f:X \to B$.

The analogous long exact sequence for DGAs can be proved by an application of
Prop.~\ref{14.4ext}.  Let $\mathcal{A} \otimes {\wedge}V$ be an $n$-dimensional
elementary extension of a minimal DGA $\mathcal{A}$, and let
$\ph:\mathcal{A} \to \mathcal{B}$ be a homomorphism.  Then
there is an exact sequence of groups
\begin{multline} \label{ES:algebras}
  \cdots \to H^{n-k}(\mathcal{B};V^*) \to
  [\mathcal{A} \otimes {\wedge V},\mathcal{B} \otimes \mathbb{F}\langle e^{(k)} \rangle]_{\tilde \ph} \to
  [\mathcal{A},\mathcal{B} \otimes \mathbb{F}\langle e^{(k)} \rangle]_\ph
  \xrightarrow{\iota_k} H^{n-k+1}(\mathcal{B};V^*) \to \cdots \\
  \cdots \to [\mathcal{A},\mathcal{B} \otimes {\wedge e^{(1)}}]_\ph
  \xrightarrow{\iota_1} H^n(\mathcal{B};V^*) \to
  \left\{\begin{array}{c}\text{homotopy classes}\\
  \text{of extensions of }\ph\end{array}\right\} \to 0.
\end{multline}
Here again $\tilde \ph:\mathcal{A} \otimes {\wedge}V \to \mathcal{B}$ is an
arbitrary extension of $\ph$, and $e^{(k)}$ represents a $k$-dimensional
generator with $de=0$ and $e^2=0$.  (Note that $\mathbb{F}\langle e^{(k)} \rangle$
is not minimal when $k$ is even.)

When $B$ is nilpotent, $\mathcal{A}$ is a minimal model for $B$,
$V=\pi \otimes \mathbb{Q}$, and $\mathcal{B}$ is geometric for $X$, there is a
homomorphism between these two sequences; in fact, by induction on elementary
extensions, when $B$ is a rational space, this homomorphism is an isomorphism.
Therefore, for any nilpotent space this homomorphism is the tensor product with
$\mathbb{Q}$, as shown for example by Sullivan as part of the proof of
\cite[Thm.~10.2(i)]{SulLong}.

The group operation on
$[\mathcal{A},\mathcal{B} \otimes \mathbb{F}\langle e^{(k)} \rangle]_\ph$ is given
as follows.  We can represent any element as $F=\ph+\eta \otimes e$, where
$\eta:\mathcal{A}^* \to \mathcal{B}^{*-k}$ satisfies the identities $d\eta=\eta d$
and
\begin{equation} \label{Leibniz}
  \eta(uv)=(-1)^{\deg v}\eta(u)\ph(v)+\ph(u)\eta(v).
\end{equation}
Then we define the operation $\boxplus$ on such elements by the formula
$$(\ph+\eta \otimes e)\boxplus(\ph+\zeta \otimes e)=\ph+(\eta+\zeta) \otimes e.$$
When we view $e$ as the volume element on $S^k$, this operation is homotopic to
the image of the usual operation in $\pi_k$ by an Eckmann--Hilton argument.  We
can then identify $\iota_k$ with
$$\ph+\eta \otimes e \mapsto \eta \circ d|_V:V \to \mathcal{B}^{n-k+1}.$$

\section{Torsion and density growth}

The finite-to-one-ness statement mentioned in the introduction was proved by
Sullivan as part of the following more general result (see
\cite[Theorem 10.2(i)]{SulLong} and its proof).
\begin{thm*}
  Let $X$ be a finite complex and $Y$ a nilpotent space of finite type (over the
  integers).  Then
  \begin{enumerate}
  \item the localization map $\rat:[X,Y] \to [X,Y_{(0)}]$ is finite-to-one;
  \item and for all $i>0$ and $f:X \to Y$, the map
    $$\pi_i(Y^X,f) \otimes \mathbb{Q} \to \pi_i((Y_{(0)})^X,f_{(0)})$$
    induced by localization is an isomorphism.
  \end{enumerate}
\end{thm*}
One might hope that the finiteness in (1) is \emph{uniform}, that is, that
cardinalities of preimages of points are bounded by some constant $N(X,Y)$.
Indeed this is obviously true when $X=S^n$, since in that case the
correspondence is a group homomorphism and each such preimage is a coset of the
kernel.  In general, however, the size of this preimage may grow without bound
depending on the rational homotopy class; we then say that $[X,Y]$ exhibits
\emph{torsion growth}.  Instead, the quantitative version of Sullivan's theorem
is provided by Theorem \ref{torGrowth}, which implies Theorems \ref{thm:tg} and
\ref{thm:g}.

In this section, we provide three examples of torsion growth and related
phenomena that motivate the rest of the discussion.
\begin{figure}
  \centering
  \begin{tikzpicture}
    \useasboundingbox (-2,-2.5) rectangle (7.6,1);
    \draw (1,1) .. controls (-2,.5) and (-.5,-1.5) .. (1,1);
    \node(S4) at (.1,.4){$S^3$};
    \draw (1,1) .. controls (.8,-2.5) and (4.5,1.2) .. (1,1);
    \node(S4) at (1.8,.3){$S^4$};
    \draw (-.65,-.18) .. controls (0,-.7) and (-.1,-1) .. (-.7,-.8);
    \draw (-.7,-.8) .. controls (-3.5,.4) and (-1.6,-3.5) .. (-.6,-1.5);
    \draw (1.8,-.5) .. controls (.2,-.8) and (0,-.8) .. (-.6,-1.5);
    \draw[->] (2,1) .. controls (2.5,1.2) and (3,1.2) .. (3.7,.8);
    \draw[->] (-.7,-1.8) .. controls (1,-2.1) and (2,-1.9) .. (3.1,-1.6);
    \draw (5,-.5) .. controls (5.2,3) and (1.5,-.5) .. (5,-.5);
    \node(S4) at (4.2,.3){$S^4$};
    \draw (5,-.5) .. controls (4.6,-4) and (1,-.5) .. (5,-.5);
    \node(S4) at (4,-1.3){$S^7$};
    \draw[->] (4.9,.8) .. controls (5.5,1) .. (6.15,.65)
      node[pos=0.5, anchor=south]{degree $p$};
    \draw[->] (4.7,-1.7) .. controls (5.1,-2) .. (6.15,-.65)
      node[pos=0.7, anchor=north west, align=center]{Hopf\\invariant $q$};
    \draw (6.8,0) circle [radius=.8] node{$S^4$};
  \end{tikzpicture}
  
  \caption{
    Construct maps $S^3 \times S^4 \to S^4$ by ``budding off'' a small ball and
    then projecting the rest onto the $S^4$ factor.
  } \label{fig:s3s4}
\end{figure}
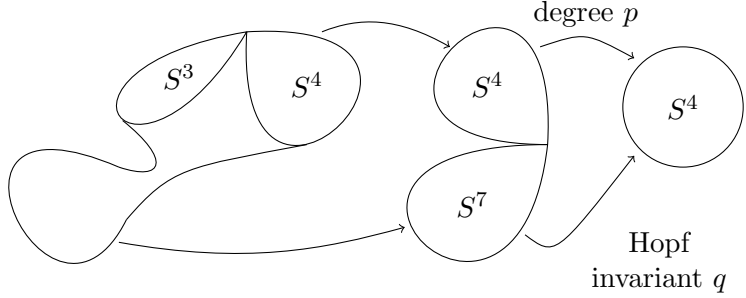
\begin{ex} \label{ex:3x4}
  By obstruction theory, the homotopy class of a map $f:S^3 \times S^4 \to S^4$
  is determined by the degree on the $S^4$ factor and a Hopf invariant on the
  top-dimensional cell.  However, some combinations determine homotopic maps.
  To see this, we fix two maps $f_1,f_2:S^3 \times S^4 \to S^4$ with degree $d$
  and Hopf invariants $h_1,h_2$ which factor through a map
  $S^3 \times S^4 \to S^7 \vee S^4$ given by pinching off a disk in the top cell
  and projecting the rest onto the $S^4$ factor, as in Figure \ref{fig:s3s4}.

  Suppose $H:S^3 \times S^4 \times I \to S^4$ is a homotopy between two such
  maps.  The original maps factor through $S^7 \vee S^4$, so such a homotopy
  factors as
  $$S^3 \times S^4 \times I \xrightarrow{H_1} U \xrightarrow{H_2} S^4$$
  where $U$ is given by collapsing each end of the cylinder to a copy of
  $S^7 \vee S^4$.  This $U$ is homotopy equivalent to $S^4 \times S^4$ minus two
  open disks; here the first $S^4$ factor is $S^3 \times I$ modulo the ends.
  Then $H_2$ sends the boundaries of the two disks to $S^4$ via maps of Hopf
  invariant $h_1$ and $h_2$, and the second $S^4$ factor via a map of degree $d$.
  It remains to determine the degree on the first $S^4$ factor, which we call
  $a$. In order for the map to be defined on the top cell, the Hopf invariant on
  its boundary must be 0, that is
  $$2ad=h_1-h_2.$$
  Thus $a$ is determined by $d$, $h_1$, and $h_2$, and such a homotopy can be
  constructed if and only if $h_1-h_2$ is an integer multiple of $2d$.  In other
  words, maps $S^3 \times S^4 \to S^4$ which have degree $d$ on the $S^4$ have a
  nontrivial Hopf invariant modulo $2d$.  For $d \neq 0$, this gives $2d$
  elements of $[X,Y]$ above a single element of $[X,Y_{(0)}]$.
\end{ex}
Repeating this example with maps $S^3 \times \mathbb{H}\mathbf{P}^n \to
\mathbb{H}\mathbf{P}^n$ gives $(n+1)d^n$ different maps of degree $d$ on the
second factor.  Thus the torsion growth of $[X,Y]$ may be an arbitrarily large
polynomial in the rational homotopy invariants.  The second part of Theorem
\ref{torGrowth} shows that this is the worst that can happen.

Notice that torsion growth occurs when the obstruction theory is affected by what
happens in lower dimensions.  A similar situation may lead to another kind of
growth which is also limited by Theorem \ref{torGrowth}.  Namely, the ``density''
of rational homotopy classes which come from genuine, integral homotopy classes
in $[X,Y]$ (\emph{integral classes} for short) may grow as we look at larger
balls in $\Hom(\mathcal{M}_Y^*,\mathcal{M}_X^*)$.
\begin{ex}
  Consider now the space $X=S^3 \times (S^4 \vee S^4)$.  Similarly to the
  previous example, elements of $[X,S^4]$ are determined by degrees $\alpha_1$
  and $\alpha_2$ on the two copies of $S^4$ and Hopf invariants $\beta_1$ and
  $\beta_2$ on the two $7$-cells.  We now translate this into rational homotopy
  theory.  The two spaces have minimal models
  $\mathcal{M}_{S^4}^*=\langle a^{(4)},b^{(7)} \mid da=0, db=a^2 \rangle$ and
  $$\mathcal{M}_X^*=\left\langle x^{(3)},y_1^{(4)},y_2^{(4)},z_{11}^{(7)},z_{12}^{(7)},
  z_{22}^{(7)}, \ldots \mid dz_{ij}=y_iy_j, \ldots \right\rangle$$
  (omitting higher-degree terms) and homomorphisms are given by
  \begin{align*}
    a &\mapsto \alpha_1y_1+\alpha_2y_2 \\
    b &\mapsto \alpha_1^2z_{11}+2\alpha_1\alpha_2z_{12}+\alpha_2^2z_{22}
    +\beta_1xy_1+\beta_2xy_2
  \end{align*}
  for any $\alpha_i,\beta_i \in \mathbb{Q}$.  However, when
  $\alpha_1\beta_2=\alpha_2\beta_1$, a homotopy to the homomorphism with
  $\beta_1=\beta_2=0$ is given by
  \begin{align*}
    a &\mapsto \alpha_1y_1+\alpha_2y_2+\frac{\beta_1}{\alpha_1}x \otimes dt \\
    b &\mapsto \alpha_1^2z_{11}+2\alpha_1\alpha_2z_{12}+\alpha_2^2z_{12}
    +(\beta_1y_1+\beta_2y_2)x \otimes t.
  \end{align*}
  Extending by linearity, we see that the representatives of a given rational
  homotopy class form a line of slope $\alpha_2/\alpha_1$ in the
  $(\beta_1,\beta_2)$-plane.  (Thus the space of homotopy classes in this plane
  is one-dimensional.)  The lines which pass through lattice points are integral.
  Thus when $\alpha_1$ and $\alpha_2$ are relatively prime, a ball of radius $R$
  in this plane contains $2\max\{\alpha_1,\alpha_2\}R$ integral
  classes.\footnote{By repeating the analysis in the previous example, one can
    see that the size of the preimage of this ball in $[X,Y]$, without
    identifying classes which are the same rationally, is
    $2\max\{\alpha_1,\alpha_2\}R \pm \gcd\{\alpha_1,\alpha_2\}$, regardless of
    what the gcd is.}  As we allow $\alpha_1$ and $\alpha_2$ to increase, this
  density grows, and the total number of integral classes in an $R$-ball in
  $\Hom(\mathcal{M}^*_Y,\mathcal{M}_X^*)$ is $\sim R^4$, rather than $\sim R^3$ as
  one may expect purely by looking at the dimension of the space of rational
  homotopy classes.\footnote{This relies on the fact that a positive fraction
    (namely, $6/\pi^2$) of all pairs of numbers are relatively prime.}
\end{ex}
Thus, Theorem \ref{torGrowth} may be rephrased as saying that both torsion growth
and density growth of $[X,Y]$ are always at worst polynomial.

Finally, we compute an example in which only looking at the volume growth of the
space of DGA homomorphisms actually yields the wrong overall bound on the growth
of $[X,Y]$.  This is in contrast with the previous two examples, where the number
of distinct maps of degree zero with at most a given Lipschitz constant, which is
determined by the Hopf invariant, swamps the ``extra'' elements coming from the
torsion and density growth.  Indeed, in this example, the correct bound is not
even polynomial!
\begin{ex}
  Let $X=(S^3 \times S^4)\#(S^3 \times S^4)$.  Using the same method as in
  Example \ref{ex:3x4}, we see the following:
  \begin{itemize}
  \item The homotopy class of a map $X \to S^4$ is determined by degrees $d_1$
    and $d_2$ on the two $S^4$ factors and a Hopf invariant $h$ on the 7-cell.
  \item The invariant $h$ is well-defined modulo $2\gcd(d_1,d_2)$.
  \end{itemize}
  We now estimate the number of homotopy classes which have representatives with
  Lipschitz constant at most $L$.  By a minimal model analysis, such a homotopy
  class must have $d_i=O(L^4)$ and (if $d_1=d_2=0$) $h=O(L^8)$.  Conversely, any
  homotopy class with $d_i \leq L^4$ and $h \leq L^8$ can be realized with
  Lipschitz constant $O(L)$ by the construction in Example \ref{ex:3x4}.  The
  number of such homotopy classes is
  $$2L^8+4\sum_{0<d<L^4} 2d+\sum_{0<|d_1|,|d_2|<L^4} 2\gcd(d_1,d_2).$$
  Clearly, the last term is asymptotically at least as large as the other two.
  Now, if $N \geq k$, then the proportion of pairs $0<a,b \leq N$ with
  $\gcd(a,b)=k$ is
  \begin{itemize}
  \item at most that of pairs for which $k$ divides both $a$ and $b$ (i.e.
    $1/k^2$);
  \item and at least
    $$\frac{1}{4}\left(\frac{1}{k^2}-\sum_{\ell=2}^\infty \frac{1}{(k\ell)^2}\right)
    =\frac{1}{4k^2}\left(2-\frac{\pi^2}{6}\right).$$
    Here the factor of $1/4$ comes from accommodating the possibility that $k$
    doesn't evenly divide $N$, and the summation from an overcount of all the
    pairs which have gcd divisible by and strictly larger than $k$.
  \end{itemize}
  Therefore, to within a multiplicative constant, the number of integral homotopy
  classes with these bounds is
  $$\sum_{0<|d_1|,|d_2|<L^4} 2\gcd(d_1,d_2) \sim \sum_{k=1}^{L^4} \frac{L^8}{k^2} \cdot
  2k \sim L^8\log L.$$
  This is therefore the growth function of $[X,S^4]$.
\end{ex}

\section{Polynomially bounded functionals on $[X,Y]$}

We now introduce some vocabulary to talk about the quantitative properties of the
fundamental correspondences in rational homotopy theory.

Let $X$ be a finite simplicial complex and $Y$ a nilpotent space of finite
$\mathbb{Q}$-homological type.  As mentioned previously, homotopy classes of maps
$X \to Y_{(0)}$ are in bijection with homotopy classes of DGA homomorphisms
$\mathcal{M}_Y^* \to A^*X$ from the minimal model of $Y$ to the simplexwise
polynomial forms on $X$.\footnote{This is Theorem 10.1(iii) in \cite{SulLong};
  it follows from the adjunction between the PL de Rham and spatial realization
  functors discussed in \cite[\S8]{BoGu}.}  So while the domain and range are
both potentially infinite-dimensional as vector spaces, there is a
finite-dimensional vector subspace of
$\Hom_{\mathbb{Q}\text{-v.s.}}(\mathcal{M}_Y^*,A^*X)$ which contains representatives
of every homotopy class, as a consequence of the following lemma.
\begin{lem} \label{lem:WExist}
  There is a finite-dimensional space $W \subset A^*X$ such that every homotopy
  class of maps $\mathcal{M}_Y^* \to A^*X$ has a representative where the images
  of all the indecomposables are in $W$.
\end{lem}
\begin{proof}
  We show this by induction on elementary extensions.  In the base case, there is
  one DGA map $\mathbb{Q} \to A^*X$, whose image is in $\mathbb{Q}$.  Now suppose
  that every homotopy class of maps $\mathcal{M}_Y^*(k) \to A^*X$ has a
  representative such that the indecomposables land in a finite-dimensional
  subspace $W_k$.  Suppose
  $\mathcal{M}_Y^*(k+1)=\mathcal{M}_Y^*(k) \otimes {\wedge V}$, where $V$ is of
  degree $n_{k+1}$.  Then for any such representative, $dV$ lands in the
  finite-dimensional subspace $D_{k+1} \subset \mathbb{Q}[W_k]$ consisting of
  $(n_{k+1}-1)$-coboundaries.  Let $S_{k+1}$ be a finite-dimensional subspace of
  $A^*X$ such that $d|_{S_{k+1}}$ is an isomorphism to $D_{k+1}$.  Finally, let
  $\mathcal{H}^{n_{k+1}}(X;\mathbb{Q})$ be a subspace of $A^*X$ containing a
  representative for each element of $H^{n_{k+1}}(X;\mathbb{Q})$.  By obstruction
  theory, every homotopy class of DGA maps $\mathcal{M}^*_Y(k+1) \to A^*X$ has a
  representative in
  $$W_{k+1}:=W_k+S_{k+1}+\mathcal{H}^{n_{k+1}}(X;\mathbb{Q}).$$
  Then we can set $W=W_r$, where $r$ is maximal such that $n_r \leq \dim X$.
\end{proof}
Note that given a minimal model $m_X:\mathcal{M}_X^* \to A^*X$ for $X$, we can
always choose $W \subset A^*X$ as in the proof to be a subset of the image of
$m_X$.  This is true even if $X$ is not nilpotent: in this case a minimal model
can still be chosen for $A^*X$, although it may be infinite-dimensional in each
degree and may not be a good homotopical model for $X$.  We fix the notation
$W(m_X,Y)$ for the vector space constructed in this way, as it will be useful
later.  We also write $\mathbb{Q}[W]$ for the subalgebra of $A^*X$ generated by
$W$.  This is still finite-dimensional in each degree and zero in degrees
$>\dim X$, allowing us to use $\Hom(\mathcal{M}_Y^*,\mathbb{Q}[W])$ as a larger
but still finite-dimensional substitute for the set of homomorphisms which send
indecomposables of $\mathcal{M}_Y^*$ to $W$.

Now let $F:[X,Y] \to \mathbb{R}$ be a functional\footnote{In the old-fashioned
 sense of a mapping assigning numerical values to elements of a function space.}.
We say that $F$ is \emph{polynomially bounded with respect to} a $W$ as above if
for some (equivalently, any) choice of norms on $W$ and
$\mathcal{M}_Y^{\leq \dim X}$, there is some $p$ such that for every
$\alpha \in [X,Y]$
\begin{equation} \label{bdAbove}
\lvert F(\alpha) \rvert=O\bigl(\left(\min\{\lVert\ph\rVert_{\text{op}} : \ph \in
\Hom(\mathcal{M}_Y^*,\mathbb{Q}[W])
\text{ with }[\ph]=\alpha_{(0)}\}\right)^p\bigr).
\end{equation}
Here $\lVert\cdot\rVert_{\text{op}}$ represents the operator norm and $\alpha_{(0)}$
is the image of $\alpha$ in $[\mathcal{M}_Y^*,A^*(X)]$.  We say that $F$ is
\emph{polynomially bounded} if it is polynomially bounded with respect to all
choices of $W$.

Likewise, if for some (not necessarily every!) $W$ the reverse inequality holds,
i.e.~for some $p>0$
\begin{equation} \label{bdBelow}
F(\alpha)=\Omega\bigl(\left(\min\{\lVert\ph\rVert_{\text{op}}: \ph \in
\Hom(\mathcal{M}_Y^*,\mathbb{Q}[W])
\text{ with }[\ph]=\alpha_{(0)}\}\right)^p\bigr),
\end{equation}
we say that the functional $F$ is polynomially bounded below.

The degree $p$ in \eqref{bdAbove} may certainly depend on the choice of $W$.  For
example, one may take $X=S^3$, $Y=S^2$, and $F$ to simply be the absolute value
of the Hopf invariant.  Here we have
$\mathcal{M}_X^*=\langle x^{(3)} | dx=0 \rangle$ and
$$\mathcal{M}_Y^*=\langle a^{(2)}, b^{(3)} | da=0, db=a^2 \rangle.$$
The obstruction-theoretic choice of $W$ as in Lemma \ref{lem:WExist} is a
purely 3-dimensional one generated by a volume form on $S^3$; the map
$a \mapsto 0,b \mapsto kd\vol_{S^3}$ is homotopic to the pullback of a map of Hopf
invariant $k$, so the bound for this choice is linear.  But we could also, for
example, choose $W=\langle f^*d\vol_{S^2}, d\vol_{S^3} \rangle$, where
$f:S^3 \to S^2$ is the Hopf fibration.  Then a map of Hopf invariant $k^2$ can be
represented by $a \mapsto kf^*d\vol_{S^2}$, $b \mapsto 0$,\footnote{Note that we
  still need to include $d\vol_{S^3}$ in $W$ in order to represent maps with Hopf
  invariant not a perfect square.} and so the polynomial bound on $F$ is only
quadratic.

It is probably the case, although we do not know of specific examples, that
functionals may be polynomially bounded with respect to some $W$ without being
polynomially bounded.\footnote{This is because there are ways to define $W$ so
  that the images of indecomposables must be related via essentially arbitrary
  systems of rational diophantine equations, and very little is known about how
  the minimal size of solutions to these depends on parameters.}  On the other
hand, it is always enough to test the specific $W$ we have already constructed.
\begin{lem} \label{lem:WBest}
  Fix a minimal model $m_X:\mathcal{M}_X^* \to A^*X$.  Whenever a functional $F$
  on $[X,Y]$ is polynomially bounded with respect to $W(m_X,Y)$, it is
  polynomially bounded.  Similarly, $F$ is polynomially bounded below if and only
  if \eqref{bdBelow} holds for $W(m_X,Y)$.
\end{lem}
When $X$ is also nilpotent, we may also choose
$\mathbb{Q}[W]=m_X(\mathcal{M}_X^{\leq \dim X})$.  In this case, it's possible to
define the degree of polynomiality of a functional $F$ as its degree with respect
to this $W$.  This may be different from the degree with respect to the Lipschitz
norm given by the minimal Lipschitz constant of a representative; in this paper
we will take the latter as more natural, as explained in the introduction.
\begin{proof}[Proof of Lemma \ref{lem:WBest}.]
  Write $W=W(m_X,Y)$.  For every $k$, fix subspaces $S_k$ and
  $\mathcal{H}^{n_k}(X;\mathbb{Q})$ of $W$ as constructed in the proof of Lemma
  \ref{lem:WExist}.

  Let $W^\prime$ be another finite-dimensional subset of $A^*X$ such that maps to
  $W^\prime$ contain representatives for all of $[X,Y]$.  It is enough to show
  that there is a polynomial $P(t)$ such that for every
  $\ph^\prime:\mathcal{M}^*_Y \to \mathbb{Q}[W^\prime]$, there is a
  $\ph:\mathcal{M}^*_Y \to \mathbb{Q}[W]$ which is homotopic to $\ph^\prime$ as a
  map to $A^*X$ such that $\lVert\ph\rVert_{\mathrm{op}} \leq
  P(\lVert\ph^\prime\rVert_{\mathrm{op}})$.  Then if $F$ is bounded with respect to
  $W$ by a polynomial $P_0$, then it is bounded with respect to $W^\prime$ by
  $P \circ P_0$.

  We show this by induction on elementary extensions.  The point is to move
  $\ph^\prime|_{\mathcal{M}_Y^*(k)}$ to a $\ph_k:\mathcal{M}_Y^*(k) \to \mathbb{Q}[W]$
  which is sufficiently nearby, in the sense that there is a polynomial-size
  homotopy from one to the other.  This allows us to lift $\ph_k$ to a $\ph_{k+1}$
  which is still not too far away.

  Formally, we keep track of the operator norm of $\ph_k$ (which sends
  indecomposables to $W(m_X,Y)$) and of the homotopy $H_k$ between
  $\ph^\prime|_{\mathcal{M}_Y^*(k)}$ and $\ph_k$ (which sends indecomposables to
  $U_k \otimes \langle t^{\leq k},dt \rangle$ for some finite-dimensional $U_k$
  which depends only on $W$ and $W'$).  At the $(k+1)$st step, lifting the
  homotopy increases these operator norms by at most a polynomial, again
  depending only on $W$ and $W'$.


  Specifically, we define $\ph_{k+1}$ and $H_{k+1}$ as follows.  Write
  $\mathcal{M}^*_Y(k+1)=\mathcal{M}^*_Y(k)\otimes{\wedge V}$; choose a finite
  subspace $S^\prime \subset A^*X$ such that $d|_{S^\prime}$ is an isomorphism to the
  space of coboundaries in the subspace
  $$S_{k+1}+\mathcal{H}^{n_{k+1}}(X;\mathbb{Q})+(U_k+W^\prime)^{n_{k+1}} \subset A^*X.$$
  By the discussion after Proposition \ref{extExist}, to extend $\ph_k$ and $H_k$
  it is enough to choose $(b,c)$ satisfying
  \begin{align*}
    db(v) &= \ph_k(dv) \\
    dc(v) &= b(v)-\ph^\prime(v)-{\textstyle\int_0^1 H_k(dv)}
  \end{align*}
  for $v \in V$.  Then we can set $\ph_{k+1}(v)=b(v)$ and
  $$H_{k+1}(v)=\ph_{k+1}(v)+{\textstyle\int_0^t H_k(dv)}+d(c(v) \otimes t).$$
  To choose $b$ and $c$ in a polynomially bounded way, first let
  $\tilde b(v)=(d|_{S_{k+1}})^{-1}(\ph_k(dv)) \in S_{k+1}$.  Then
  $$\tilde b(v)-\ph^\prime(v)-{\textstyle\int_0^1 H_k(dv)}$$
  is a cycle in $S_{k+1}+U_k+W^\prime$; let $a(v)$ be the representative of its
  homology class in $\mathcal{H}^{n_{k+1}}(X;\mathbb{Q})$.  Then we choose
  $b(v)=\tilde b(v)-a(v)$ and $c(v)$ to be the antiderivative in $S^\prime$ of
  $$\tilde b(v)-a(v)-\ph^\prime(v)-{\textstyle\int_0^1 H_k(dv)}.$$
  All four of these terms are polynomially bounded in terms of $\ph_k$ and $H_k$,
  with the polynomial depending only on the differential on $\mathcal{M}^*_Y$ and
  the structure of $W$ and $W'$.
\end{proof}
All the above results, starting with Lemma \ref{lem:WExist}, also hold for the de
Rham algebra $\Omega^*X$; in this case one should talk of $\mathbb{R}[W]$ rather
than $\mathbb{Q}[W]$.

We now give the two main examples of polynomially bounded functionals which
motivate the definitions.
\begin{thm} \label{lipPBB}
  Suppose that $X$ and $Y$ are finite complexes with piecewise Riemannian
  metrics, $Y$ nilpotent.  Then the functional $\Lip:[X,Y] \to \mathbb{R}^+$
  given by
  $$\Lip \alpha=\inf \{\Lip f:
  f\text{ is a Lipschitz representative of }\alpha\}$$
  is polynomially bounded below.
\end{thm}
\begin{thm} \label{torGrowth}
  For finite complexes $X$ and $Y$ with $Y$ nilpotent, and for any
  $W \subset A^*X$ as in the statement of Lemma \ref{lem:WExist}, the number of
  homotopy classes in $[X,Y]$ whose image in $[\mathcal{M}_Y,A^*X]$ has a
  representative in the $R$-ball in $\Hom(\mathcal{M}^*_Y,W)$ is bounded by a
  polynomial in $R$.  In particular, the functional
  $\#:[X,Y_{(0)}] \to \mathbb{R}^+$, which measures the size of the preimage in
  $[X,Y]$ of each class under composition with the rationalization map, is
  polynomially bounded.
\end{thm}
These two results combine to immediately yield Gromov's conjecture.
\begin{cor}
  If $X$ and $Y$ are finite complexes and $Y$ is nilpotent, the number of
  homotopy classes of maps $X \to Y$ which have representatives with Lipschitz
  constant at most $L$ is bounded by a polynomial in $L$.
\end{cor}
The proof of Theorem \ref{torGrowth} is deferred to Section \ref{S:QFin}.

Before we begin the proof of Theorem \ref{lipPBB}, we remark that since we have
not used the equivalence between the rational homotopy categories of spaces and
DGAs, everything discussed thus far in this section is true for real DGAs as well
as rational ones.  In other words, if $X$ is a smooth manifold, perhaps with
boundary, we can replace $A^*X$ with $\Omega^*X$ without changing any of the
arguments.
\begin{proof}[Proof of Theorem \ref{lipPBB}.]
  We first note that this property is invariant under Lipschitz homotopy
  equivalence.  Therefore we can assume that $X$ and $Y$ are compact Riemannian
  manifolds with boundary, by embedding them in some high-dimensional Euclidean
  space and thickening.

  From here, the proof follows the same outline as that of Lemma \ref{lem:WBest}.
  Let $f:X \to Y$ be a map with Lipschitz constant $L$.  Fix real minimal models
  $m_X:\mathcal{M}_X^* \to \Omega^*X$ and $m_Y:\mathcal{M}_Y^* \to \Omega^*Y$
  (here again the fact that $\mathcal{M}_X^*$ may not be a good homotopical model
  is immaterial).  We would like to show that for some polynomial $P$ depending
  only on $m_Y$ and $m_X$, $f^*m_Y:\mathcal{M}_Y^* \to \Omega^*X$ is homotopic to
  a map $\ph:\mathcal{M}_Y^* \to W(m_X,Y)$ such that $\lVert\ph\rVert \leq P(L)$.

  This is proved by induction on the minimal model of $Y$.  We let the
  $\infty$-norm be our chosen norm on $\Omega^*X$.  This gives us an operator
  norm on maps $\mathcal{M}^*_Y \to \Omega^*X$ which is well-defined up to a
  constant.  Notice that the Lipschitz condition implies that for
  $\omega \in \Omega^n(Y)$,
  $$\lVert f^*\omega \rVert_\infty \leq L^n\lVert \omega \rVert_\infty.$$
  In other words, $\lVert f^*m_Y \rVert_{\text{op}} \leq C(\dim X,m_Y)L^{\dim X}$.

  At each stage of the induction we produce maps
  $\ph_k:\mathcal{M}_Y^*(k) \to W(m_X,Y)$ and $H_k:\mathcal{M}_Y^*(k) \to
  \Omega^*X \otimes \wedge(t,dt)$ such that:
  \begin{itemize}
  \item $\lVert \ph_k \rVert \leq P_k(L)$;
  \item if we write
    $$H_k(y)=\sum_{i=0}^r I_{k,i}(y) \otimes t^i
    +\sum_{j=0}^s J_{k,j}(y) \otimes t^jdt,$$
    then $r$ and $s$ depend only on $X$ and $Y$,
    $\lVert I_{k,i} \rVert_{\text{op}} \leq P_k(L)$, and
    $\lVert J_{k,j} \rVert_{\text{op}} \leq P_k(L)$.
  \end{itemize}
  The induction step proceeds exactly as in Lemma \ref{lem:WBest}, except that
  since $f^*m_Y$ is not guaranteed to land in a finite subspace, neither can we
  guarantee this about $H_k$.  Therefore, we also cannot choose antiderivatives
  from a finite subspace.  Instead we use the following lemma, which dates back
  to \cite{GrDil} and is proved carefully for simplicial complexes as
  \cite[Lemma 2--2]{PCDF}:
  \begin{lem}[Coisoperimetric inequality] \label{lem:coIP}
    Given a compact Riemannian manifold $X$, there is a constant $I(X,n-1)$ such
    that any exact form $\beta \in \Omega^n(X)$ has an antidifferential
    $\alpha \in \Omega^{n-1}(X)$ with $\lVert\alpha\rVert_\infty \leq I(X,n-1)
    \lVert\beta\rVert_\infty$.
  \end{lem}
  We then produce $\ph_{k+1}$ and $H_{k+1}$ as before.  Write
  $\mathcal{M}^*_Y(k+1)=\mathcal{M}^*_Y(k) \otimes {\wedge V_{k+1}}$.  First we
  produce $\tilde b(v)$ for each element $v$ of a basis for $V_{k+1}$.  This
  gives a cycle
  $$\tilde b(v)-f^*m_Y(v)-{\textstyle\int_0^1 H_k(dv)}$$
  whose $\infty$-norm is polynomially bounded in $L$; this also gives a bound on
  its cohomology class, obtained by integrating it against cycles generating
  $H_{n_{k+1}}(X;\mathbb{R})$.  Thus we can choose $a(v)$ and $b(v)$ as before.
  Applying the coisoperimetric inequality to
  $$\tilde b(v)-a(v)-f^*m_Y(v)-{\textstyle\int_0^1 H_k(dv)}$$
  we get a polynomially bounded $c(v)$ and finally obtain $\ph_{k+1}$ and
  $H_{k+1}$ which also have polynomial estimates.
\end{proof}

\section{Quantitative finiteness} \label{S:QFin}

We now demonstrate Theorem \ref{torGrowth}.  Sullivan's result is proven by using
obstruction-theoretic exact sequences and the five lemma; for the quantitative
version, we will develop some quantitative homological algebra.
\begin{defn}
  Let $h:A \to V$ be a homomorphism from a finitely generated group to a normed
  $\mathbb{Q}$-vector space.  We say that $h$ is
  \begin{itemize}
  \item \emph{$C$-injective} if for every 1-ball $B$ in $V$,
    $\#h^{-1}(B) \leq C$;
  \item \emph{$C$-surjective} if every point of $V$ is within $C$ of $h(A)$.
  \end{itemize}
\end{defn}
\begin{lem}[Quantitative four lemmas]
  Suppose that
  $$\xymatrix{
    A_1 \ar[r]^{f_1} \ar[d]^{\ph_1} & A_2 \ar[r]^{f_2} \ar[d]^{\ph_2} &
    A_3 \ar[r]^{f_3} \ar[d]^{\ph_3} & A_4 \ar[d]^{\ph_4} \\
    V_1 \ar[r]^{m_1} & V_2 \ar[r]^{m_2} & V_3 \ar[r]^{m_3} & V_4
  }$$
  are exact sequences with $A_i$ finitely generated groups and $V_i$
  finite-dimensional normed $\mathbb{Q}$- or $\mathbb{R}$-vector spaces, such
  that $m_1$ and $m_3$ have operator norm $\leq 1$.  Let $\tau$ be a constant
  such that $m_2$ satisfies
  $$\min\{\lVert u \rVert: u \in m_2^{-1}(v)\} \leq \tau\lVert v \rVert
  \text{ for every }v \in m_2(V_2).$$
  \begin{enumerate}
  \item If $\ph_2$ is $C_2$-injective, $\ph_4$ is $C_4$-injective and $\ph_1$ is
    $C_1$-surjective, then $\ph_3$ is
    $(C_1+\tau)^{\rk m_1}\tau^{\rk m_2}C_2C_4$-injective.
  \item If $\ph_1$ is $C_1$-surjective, $\ph_3$ is $C_3$-surjective and $\ph_4$
    is $C_4$-injective, then $\ph_2$ is
    $(C_1+3\tau C_3^{\rk m_3+1}C_4)$-surjective.
  \end{enumerate}
\end{lem}
We remark that the groups $A_i$ are not necessarily abelian, although the $\ph_i$
of course factor through the abelianization map.
\begin{proof}
  We use a quantitative version of the usual diagram chasing arguments for
  proving the four lemmas.

  For the injectivity four lemma, we would like to show that for every 1-ball $B$
  in $V_3$,
  $$\#\ph_3^{-1}(B) \leq (C_1+\tau)^{\rk m_1}\tau^{\rk m_2}C_2C_4.$$
  First notice that $\#f_3(\ph_3^{-1}(B)) \leq C_4$.  Thus it is enough to show
  that for any $a \in A_4$,
  $$\#(f_3^{-1}(a) \cap \ph_3^{-1}(B)) \leq (C_1+\tau)^{\rk m_1}\tau^{\rk m_2}C_2.$$
  By shifting the center of $B$ by $-\ph_3(\tilde a)$ where $\tilde a$ is an
  arbitrary preimage of $a$, we see that it is enough to show this for $a=\id$.
  To do that, we will show that every element in $\ph_3^{-1}(B) \cap \ker f_3$ has
  a preimage in $A_2$ which lands within distance $\tau$ of a $C_1$-ball
  $\tilde B$ in a $\rk m_1$-dimensional affine subspace $v_2+m_1(V_1)$.

  Choose the center of $\tilde B$ to be an arbitrary preimage $v_2$ of the center
  of $B$.  Given $b \in \ph_3^{-1}(B) \cap \ker f_3$, choose a preimage
  $\tilde b \in A_2$: we know $\ph_2(\tilde b)$ is at most distance $\tau$ from
  $v_2+m_1(V_1)$.  Then we can choose $s \in A_1$ such that $m_1 \circ \ph_1(s)$
  is at most distance $C_1$ from $\ph_2(\tilde b)-v_2$, and therefore
  $\tilde b-f_1(s)$ is the preimage we're looking for.  This completes the proof
  of the injectivity lemma.

  For the surjectivity lemma, choose $v \in V_2$; we would like to show that
  there is an $a \in A_2$ such that $\ph_2(a)$ is contained in a
  $(C_1+3\tau C_3^{\rk m_3+1}C_4)$-ball around $v$.  For this, we will show that
  there is an element $b \in A_2$ such that
  $$\lVert m_2(v-\ph_2(b)) \rVert \leq 3C_3^{\rk m_3+1}C_4.$$
  We can find a point $v' \in \ker m_2$ whose distance from $v-\ph_2(b)$ is at
  most $3\tau C_3^{\rk m_3+1}C_4$.  Then there is an $\tilde a \in A_1$ such that
  $m_1 \circ f_1(\tilde a)$ is within $C_1$ of $v'$ and we can use
  $a=b+f_1(\tilde a)$.

  It remains to find $b$.  If $\lVert m_2(v) \rVert \leq 3C_3^{\rk m_3+1}C_4$, we
  can use $b=0$.  Otherwise, we show by induction that we can reduce to this
  case.  Let $N=C_3^{\rk m_3}C_4$, and consider the $N+1$ disjoint $C_3$-balls
  $B_i$ around $\frac{i}{N}m_2(v)$, $i=0,1,\ldots,N$.  Each of these has a
  preimage point $c_i \in A_3$; moreover, $\ph_4 \circ f_3$ sends each of the
  $c_i$ to the $C_3$-ball around zero in $V_4$, which means that they have at
  most $C_3^{\rk m_3}C_4$ distinct images under $f_3$.  By the pigeonhole
  principle, there are $i<j$ such that some $c_i \in \ph_3^{-1}(B_i)$ and
  $c_j \in \ph_3^{-1}(B_j)$ have $f_3(c_i)=f_3(c_j)$.  Then $c_j-c_i=f_2(b')$ for
  some $b' \in A_2$.  Moreover,
  $$\lVert m_2(v-\ph_2(b')) \rVert \leq \lVert m_2(v) \rVert-C_3.$$
  Now we repeat this process with $m_2(v-\ph_2(b'))$; after a finite number of
  steps, we get an element of length at most $3C_3^{\rk m_3+1}C_4$, and can set $b$
  to be the sum of all the $b'$s used along the way.
\end{proof}
We are now ready to prove the theorem along with the following extra statements:
\begin{lem} \label{injSur}
  Let $X$ and $Y$ be finite complexes with $Y$ nilpotent.  Then:
  \begin{enumerate}[label={(\roman*)}]
  \item For every $k$, there is a polynomial $P$ such that the rationalization
    map
    $$\pi_1((Y_k)^X,f) \to
    [\mathcal{M}^*_Y(k),A^*X \otimes {\wedge e^{(1)}}]_{f^*m_Y}$$
    is $P(\Lip f)$-surjective, where the norm on the latter is given by the
    operator norm on the indecomposables,
    $$\lVert \gamma \rVert_l=\inf \Bigl\{\max_{i \leq k}
    \lVert \eta|_{V_i} \rVert_{\mathrm{op}} \Bigm| \eta:\mathcal{M}_Y^*(k) \to A^*X
    \text{ s.t.~}[f^*m_Y+\eta \otimes e]=\gamma\Bigr\}.$$
  \item For every $k$, there is a polynomial $P'_k$ such that the map
    $$\left\{\begin{array}{c}\text{homotopy classes of}\\
    \text{lifts of $f_{(k-1)}$ to }Y_k\end{array}\right\} \to
    \left\{\begin{array}{c}\text{homotopy classes of}\\
    \text{extensions of $f^*m_Y(k-1)$ to }\mathcal{M}^*_Y(k)\end{array}\right\}$$
    is $P_k'(\Lip f)$-injective, where the norm on the set of extensions is
    induced by the obstruction in $H^k(X;\pi_k(Y) \otimes \mathbb{Q})$ to
    homotoping to some fixed extension.
  \end{enumerate}
\end{lem}
\begin{rmk}
  A similar proof, applied to a different portion of the long exact sequence,
  simultaneously proves that
  $$\pi_i(Y^X,f) \to \pi_i((Y_{(0)})^X,f_{(0)})$$
  is $P(\Lip f)$-injective and
  $$\pi_{i+1}(Y^X,f) \to \pi_{i+1}((Y_{(0)})^X,f_{(0)})$$
  is $P(\Lip f)$-surjective, for norms similar to those in Lemma \ref{injSur}(i).
  Thus we recover quantitative versions of the entirety of Sullivan's result.
\end{rmk}
\begin{proof}
  Write $Y$ as an inverse limit of a tower of spaces
  $$\cdots \to Y_k \to Y_{k-1} \to \cdots \to Y_0=*$$
  where each $Y_k \to Y_{k-1}$ is a principal $K(A_k,n_k)$-fibration,
  $n_k \geq n_{k-1}$.  Fix $W$ as in the previous section; let
  $\ph:\mathcal{M}_Y^* \to A^*X$ be a homomorphism which sends indecomposables to
  $W$ and is homotopic to $f^*m_Y$.

  Our goal here is to understand the behavior of $[X,Y]$, that is, $\pi_0$ of the
  mapping space $Y^X$.  To do this, we need to also consider the behavior of
  $\pi_1((Y_k)^X,f)$ at various stages $k$ and with various basepoints $f$, and
  its rationalization homomorphism to
  $$\Pi(k,\ph):=[\mathcal{M}^*_Y(k),A^*X \otimes {\wedge e^{(1)}}]_\ph.$$

  By induction on $k$, we will construct polynomials $P_k$ such that, for the
  norm in the statement of the lemma, the homomorphisms
  $$\pi_1((Y_k)^X,f) \to \Pi(k,\ph)$$
  are $P_k(\lVert\ph\rVert)$-surjective.  In turn, we will use this to construct
  polynomials $P^\prime_k$ such that the homomorphisms
  $$\left\{\begin{array}{c}\text{homotopy classes of}\\
  \text{lifts of $f$ to }Y_k\end{array}\right\} \to
  \left\{\begin{array}{c}\text{homotopy classes of}\\
  \text{extensions of $\ph$ to }\mathcal{M}^*_Y(k)\end{array}\right\}$$
  are $P'_k(\lVert\ph\rVert)$-injective, where the set of extensions is given a
  group structure by fixing a basepoint.  Then the number of classes in $[X,Y]$
  which map to the $R$-ball of $\Hom(\mathcal{M}^*_Y,\mathbb{Q}[W])$ is at most
  \begin{equation} \label{ballBound}
    R^{\sum_{k=1}^r \dim H^{n_k}(X;A_k \otimes \mathbb{Q})}\prod_{k=1}^r P^\prime_k(R),
  \end{equation}
  where $r=\max \{i: n_i \leq \dim X\}$.  This is the estimate we are looking
  for (although it may often be a drastic overcount).  The Lipschitz estimates
  then follow from Theorem \ref{lipPBB}.

  We now produce the polynomials $P_k$.  Of course, we can take $P_0=0$, since
  both groups are trivial.  For general $k$, we inductively apply the
  surjectivity four lemma to the subsequence
  \begin{equation} \label{ES:spaces2}
    H^{n_k-1}(X;A_k) \to \pi_1((Y_k)^X,\tilde f) \to \pi_1((Y_{k-1})^X,f) \to
    H^{n_k}(X;A_k)
  \end{equation}
  of the exact sequence \eqref{ES:spaces} and the corresponding subsequence
  \begin{equation} \label{ES:algebras2}
    H^{n_k-1}(A^*X;A_k \otimes \mathbb{Q}) \to \Pi(k,\ph) \xrightarrow{\rho}
    \Pi(k-1,\ph) \xrightarrow{\iota_1} H^{n_k}(A^*X;A_k \otimes \mathbb{Q})
  \end{equation}
  of \eqref{ES:algebras}.  To do this, we must put norms on the vector spaces in
  \eqref{ES:algebras2} that satisfy the relevant compatibility conditions.  Note
  that the groups in \eqref{ES:spaces2} are finitely generated as noted by
  Sullivan and perhaps already Serre.

  First, let $\hat W \supseteq W$ be a finite-dimensional subspace such that
  every homotopy class of homomorphism
  $\mathcal{M}_Y^* \to A^*X \otimes {\wedge e^{(1)}}$ has a
  representative which lands in $\mathbb{Q}[\hat W] \otimes {\wedge e}$.
  Such a subspace can be found by the method of Lemma \ref{lem:WExist}.  We put
  norms on each of the groups $\Hom(A_k,\mathbb{Q})$ ($V_k$ for short) and on the
  degree $\leq \dim X$ vectors in $\mathbb{Q}[\hat W]$.  This gives a
  well-defined operator norm on maps $V_k \to \mathbb{Q}[\hat W]$ with
  fixed-degree image---for example, on cochains that land in $\mathbb{Q}[\hat W]$.

  Now for $i=n_k$ and $n_k-1$ we define norms on $H^i(A^*X;V_k^*)$ by minimizing
  over cochain representatives that land in $\mathbb{Q}[\hat W]$:
  $$\lVert\alpha\rVert_H=\inf\left\{\lVert\psi\rVert_{\text{op}} \bigm| \psi:V_k
  \to \mathbb{Q}[\hat W]\text{ s.t.~}[\psi]=\alpha\right\}.$$
  Similarly, for $\gamma \in \Pi(k,\ph)$ we take the minimum over representatives
  of $\gamma$ of the operator norm on indecomposables, which we call the ``left
  norm'':
  $$\lVert \gamma \rVert_l=\inf \Bigl\{\max_{i \leq k}
  \lVert \eta|_{V_i} \rVert_{\text{op}} \Bigm| \eta:\mathcal{M}_Y^*(k) \to
  \mathbb{Q}[\hat W]\text{ s.t.~}[\ph+\eta \otimes e]=\gamma\Bigr\}.$$
  Finally, for $\gamma \in \Pi(k-1,\ph)$ we need to use a ``right norm'' which
  combines the left norm (for $\mathcal{M}_Y^*(k-1)$) and the homology norm on
  the image under $\iota_1$:
  $$\lVert \gamma \rVert_r=\inf \biggl\{\max\Bigl\{
  \max_{i \leq k-1} \lVert \eta|_{V_i} \rVert_{\text{op}},
  \lVert(\eta \circ d)|_{V_k}\rVert_{\mathrm{op}}\Bigr\} \biggm|
  \begin{array}{l}\eta:\mathcal{M}_Y^*(k-1) \to \mathbb{Q}[\hat W]\\
    \text{ s.t.~}[\ph+\eta \otimes e]=\gamma\end{array}\biggr\}.$$
  It is easy to see that under these norms, the outer two maps of
  \eqref{ES:algebras2} are norm-nonincreasing.  Moreover, for the restriction map
  $$\rho:(\Pi(k,\ph),\lVert\cdot\rVert_l)\to (\Pi(k-1,\ph),\lVert\cdot\rVert_r)$$
  there is a constant $\tau_k$, which is determined by the differentials on $V_k$
  and $\mathbb{Q}[\hat W]$ and is therefore independent of $\ph$, such that
  $$\min\{\lVert\tilde\gamma\rVert_l: \tilde\gamma \in \rho^{-1}(\gamma)\} \leq
  \tau\lVert\gamma\rVert_r \text{ for every }\gamma \in \rho(\Pi(\ph,k)).$$

  Finally, in order to induct we need to compare the left and right norms on
  $\Pi(k-1,\ph)$.  Indeed, there is a polynomial $Q_k$ such that for
  $\gamma \in \Pi(k-1,\ph)$,
  $$\lVert\gamma\rVert_r\leq Q_k(\lVert\ph\rVert)\cdot\lVert\gamma\rVert_l.$$
  This is because for any $u=dv$, $v \in V_k$, $\eta(u)$ decomposes by repeated
  applications of \eqref{Leibniz} as
  $$\eta(u)=\sum_i \ph(u_i)\eta(y_i),$$
  where the $y_i$ are indecomposable.  This bounds $\eta|_{d(V_k)}$ in terms of
  $\eta$ applied to indecomposables.

  Therefore, by the surjectivity four lemma,
  $$P_k(\lVert\ph\rVert) \leq C_{k,-1}+3\tau_kC_{k,0}(Q_k(\lVert\ph\rVert)
  P_{k-1}(\lVert\ph\rVert))^{\rk\iota_1+1}$$
  where $C_{k,-1}$ and $C_{k,0}$ depend only on $X$ and $Y$.  Now we apply the
  injectivity four lemma to the sequence
  $$\pi_1((Y_{k-1})^X,f) \to H^{n_k}(X;A_k) \to \{\text{lifts of }f\} \to 0$$
  and the corresponding sequence of vector spaces, letting the norm on the set of
  lifts be induced by that on $H^{n_k}(X;V_k^*)$.  We get that the rationalization
  on the set of lifts is $P_k'(\lVert\ph\rVert)$-injective where
  $$P_k'(\lVert\ph\rVert)=C_{k,0}(Q_k(\rVert\ph\rVert)
  P_{k-1}(\lVert\ph\rVert)+1)^{\rk\iota_1}$$
  where $C_{k,0}=C_{k,0}(X,Y)$ is the same as above.  Now, distances under this
  norm are a lower bound for distances under the operator norm on
  $\Hom(\mathcal{M}_Y^*,\mathbb{Q}[W])$; this proves the bound \eqref{ballBound}
  and the theorem.
\end{proof}

\section{Rational invariance}

In this section we prove the statements about rational invariance given in
\S\ref{S:ratinv}.  We first restate Theorem \ref{thm:posw}:
\begin{thm*}
  Let $X$ and $Y$ be finite metric complexes with $Y$ simply connected.  If $Y$
  (resp.~$X$) is a \emph{space with positive weights}, then the asymptotic
  behavior of $g_{[X,Y]}$ and $\mathrm{tg}_{[X,Y]}$ depends only on the rational
  homotopy type of $Y$ (resp.~$X$).
\end{thm*}
A simply connected space $Y$ has ($\mathbb{Q}$-)\emph{positive weights} (see
\cite{BMSS} or \cite{PWHT}) if the indecomposables of its minimal DGA split as
$U_1 \oplus U_2 \oplus \cdots \oplus U_r$ so that for every $t \in \mathbb{Q}$
there is an automorphism $\ph_t$ sending $v \mapsto t^iv$, $v \in U_i$.  Examples
include formal spaces \cite{Shiga}, coformal spaces \cite{PWHT}, as well as
homogeneous spaces and other spaces whose indecomposables split as
$V_0 \oplus V_1$, where $dV_0=0$ and $dV_1 \subset \bigwedge V_0$.  In
particular, the spaces in the Examples section all have positive weights.  The
lowest-dimensional nonexample, as far as we know, is a complex given in \cite{MT}
which is constructed by attaching a 12-cell to $S^3 \vee \mathbb{C}\mathbf{P}^2$;
other, much higher-dimensional non-examples are given in \cite{ArLu} and
\cite{Am}.
\begin{proof}
  Suppose $Y$ and $Y'$ are rationally equivalent simply connected finite
  complexes with positive weights.  This implies \cite{BMSS} that these spaces
  are \emph{$0$-universal}, in particular, there are maps
  $$Y \xrightarrow{\ph} Y' \xrightarrow{\psi} Y$$
  inducing rational equivalences.  We can assume that these maps are Lipschitz;
  moreover, by Prop.~\ref{prop:Fto1}, there are constants $C(\ph,X)$ and
  $C'(\psi,X)$ such that the maps $[X,Y] \to [X,Y'] \to [X,Y]$ induced by $\ph$
  and $\psi$ are, respectively, $C$-to-one and $C'$-to-one.  Then
  we immediately see that for any $X$,
  \begin{align*}
    g_{[X,Y]}(L) &\leq C g_{[X,Y']}(\Lip(\ph) \cdot L)
    \leq C' C g_{[X,Y]}(\Lip(\psi) \Lip(\ph)\cdot L) \\
    \mathrm{tg}_{[X,Y]}(L) &\leq C \mathrm{tg}_{[X,Y']}(\Lip(\ph) \cdot L)
    \leq C' C \mathrm{tg}_{[X,Y]}(\Lip(\psi) \Lip(\ph)\cdot L). \\
  \end{align*}
  Since all these functions are polynomial, this means that they are within a
  multiplicative constant of each other.

  A similar argument works for rationally equivalent $X$ and $X'$ with positive
  weights.
\end{proof}
It remains to prove Prop.~\ref{prop:Fto1}, which we again restate:
\begin{prop*}
  Given a rational homotopy equivalence $\ph:Y \to Z$ between finite nilpotent
  complexes, for any finite complex $X$, the induced map $[X,Y] \to [X,Z]$ is
  uniformly finite-to-one; i.e., preimages of classes have size bounded by some
  $C(\ph,X)$.
\end{prop*}
For the second part of Theorem \ref{thm:posw}, that concerned with the domain, we
will also need the following dual statement:
\begin{prop*}
  Given a map $\ph:X \to X'$ between finite complexes where the relative homology
  groups $H^*(X',X)$ are finite, for any simply connected finite complex $Y$, the
  induced map $[X',Y] \to [X,Y]$ is uniformly finite-to-one.
\end{prop*}
\begin{proof}[Proof of both propositions.]
  To bound the size of the preimage of a homotopy class, we use obstruction
  theory on the relative Postnikov tower
  \begin{center}
    \begin{tikzpicture}
      \node (x) at (0,0) {$Y$};
      \node (x0) at (3,0) {$P_1$};
      \node (x0prime) at (4.1,0) {$=P_0=Z$};
      \node (x1) at (3,1) {$P_2$};
      \node (vd) at (3,2) {$\vdots$};
      \node (xn) at (3,3) {$P_n$};
      \draw[->] (x) -- (x0) node[near end,anchor=south,inner sep=1pt]
           {$\ph_0=\ph$};
      \draw[->] (x) -- (x1) node[midway,anchor=south,inner sep=2pt] {$\ph_2$};
      \draw[->] (x) -- (xn) node[midway,anchor=south east,inner sep=1pt]
           {$\ph_n$};
      \draw[->] (xn) -- (vd) node[midway,anchor=west] {$p_n$};
      \draw[->] (vd) -- (x1) node[midway,anchor=west] {$p_3$};
      \draw[->] (x1) -- (x0) node[midway,anchor=west] {$p_2$};
    \end{tikzpicture}
  \end{center}
  of the map $\ph:Y \to Z$.  Here, $P_k$ is a space such that $\pi_i(P_k,Y)=0$
  for $i \leq k$ and $\pi_i(Z,P_k)=0$ for $i>k$.  The map $p_k$ therefore only
  has one nonzero (and finite) relative homotopy group, $\pi_k(Z,Y)$.  This means
  that the obstruction to homotoping two lifts of a map $X \to P_k$ to $P_{k+1}$
  lies in $H^k(X;\pi_k(Z,Y))$, which is again finite.  Thus there are at most
  $$\prod_{k=1}^{\dim X} \lvert H^k(X;\pi_k(Z,Y)) \rvert$$
  homotopy classes of maps $X \to Y$ going to any homotopy class of maps
  $Y \to Z$.

  For the dual proposition, we can use the dual argument to show that the size of
  the preimage is bounded by
  $$\prod_{k=1}^{\dim X} \lvert H^k(X',X;\pi_k(Y)) \rvert,$$
  which is also finite.
\end{proof}

\bibliographystyle{amsalpha}
\bibliography{embeddings}
\end{document}